\documentclass[12pt,a4paper,oneside]{amsart}
\usepackage{amsfonts, amsmath, amssymb, amsthm, amscd}
\usepackage{anysize}
\usepackage{mathtools}
\marginsize{2cm}{2cm}{1.5cm}{1.5cm}

\usepackage[utf8]{inputenc}
\usepackage[english]{babel}
\usepackage{cmap}
\usepackage{enumerate}

\usepackage{xcolor}
\usepackage{hyperref}
\hypersetup{
	colorlinks   = true, 
	urlcolor     = blue, 
	linkcolor    = blue, 
	citecolor    = red 
}
\newcommand{\Href}[2]{\hyperref[#2]{#1~\ref{#2}}}
\usepackage{anysize}
\marginsize{2cm}{2cm}{1.5cm}{1.5cm}

\makeatletter
\@namedef{subjclassname@2020}{%
  \textup{2020} Mathematics Subject Classification}
\makeatother

\def\R{{\mathbb R}}
\def\diam{\mathop{\rm diam}}
\newcommand{\Red}{\R^d}%
\newcommand{\vol}[1]{\operatorname{vol}\nolimits_{#1}}%
\newcommand{\la}{\lambda}
\def\co{\mathop{\rm{conv}}}
\newcommand{\iprod}[2]{\left\langle#1,#2\right\rangle}%
\newcommand{\cof}[1]{\co \! \left\{#1\right\}}
\def\polar{\circ}
\newcommand{\st}{:\;}
\providecommand{\abs}[1]{\lvert#1\rvert}%
\newcommand{\di}{\,\mathrm{d}}
\DeclareMathOperator{\tr}{\mathrm{trace}}
\newcommand{\id}{\mathrm{Id}}

\numberwithin{equation}{section}

\newtheorem{thm}{Theorem}

\newtheorem{lem}{Lemma}[section]

\newtheorem{conj}{Conjecture}[section]

\title{A Quantitative Helly-type Theorem: Containment in a Homothet}
\author{Grigory Ivanov\address{Grigory Ivanov: 
Institute of Science and Technology Austria (IST Austria), 
Kleusteneuburg, 3400, Austria; Laboratory of Combinatorial and Geometrical Structures, Moscow Institute of Physics and Technology, Moscow, 141701, Russia}
\email{grimivanov@gmail.com}
\and
M\'arton Nasz\'odi\address{M\'arton Nasz\'odi:
MTA-ELTE Lend\"ulet Combinatorial Geometry Research Group;
Dept. of Geometry, Lor\'and E\"otv\"os University, Budapest}
\email{marton.naszodi@math.elte.hu}
}

\thanks{G.I.~acknowledges the financial support from the Ministry of Educational and Science
of the Russian Federation in the framework of MegaGrant no 075-15-2019-1926.
\\
M.N. was supported by the National Research, Development and Innovation Fund (NRDI) grants K119670 and KKP-133864 as
well as the Bolyai Scholarship of the Hungarian Academy of Sciences and the New National
Excellence Programme and the TKP2020-NKA-06 program provided by the NRDI.
}


\begin{document}
\begin{abstract}
We introduce a new variant of quantitative Helly-type theorems: the minimal \emph{``homothetic distance''} of the intersection of a family of convex sets to the intersection of a subfamily of a fixed size. As an application, we establish the following quantitative Helly-type result for the \emph{diameter}. If $K$ is the intersection of finitely many convex bodies in $\mathbb{R}^d$, then one can select $2d$ of these bodies whose intersection is of diameter at most $(2d)^3\mathrm{diam}(K)$.
The best previously known estimate, due to Brazitikos, is $c d^{11/2}$. Moreover, we confirm that the multiplicative factor $c d^{1/2}$ conjectured by B\'ar\'any, Katchalski and Pach cannot be improved.
\end{abstract}
\maketitle

\section{Introduction}
In \cite{barany1982quantitative} (see also \cite{barany1984helly}), B\'ar\'any, Katchalski and Pach proved the following two statements. According to the \textbf{Quantitative Volume Theorem}, 
{\it
 if the intersection of a family of convex sets in $\Red$ is of volume one, then the intersection of some subfamily of size at most $2d$ is of volume at most $v(d)$, a constant depending only on $d$.
}
The \textbf{Quantitative Diameter Theorem} states that
{\it
 if the intersection of a family of convex sets in $\Red$ is of diameter one, then the intersection of some subfamily of size at most $2d$ is of diameter at most $\delta(d)$, a constant depending only on $d$.
}

In \cite{barany1982quantitative}, B\'ar\'any, Katchalski and Pach established an upper bound of roughly $d^{d^2}$ on $v(d)$ and conjectured that $v(d) \leq (d)^{cd}$ holds for a constant $c > 0$. 
Nasz\'odi confirmed this conjecture in \cite{naszodi2016proof} using contact points of the John ellipsoid \cite{John} of the intersection of the family of convex sets. The current best bound, $v(d)\leq(cd)^{3d/2}$, is due to Brazitikos \cite{brazitikos2017brascamp}.

In \cite{barany1982quantitative}, the authors obtained a bound on $\delta(d)$ which is exponentional in the dimension, and formulated the following conjecture.
\begin{conj}[B\'ar\'any, Katchalski, Pach \cite{barany1982quantitative}]\label{conj:BKPdiam}
 \[\delta(d)\leq c \sqrt{d}\]
 with a universal constant $c>0$.
\end{conj}

 Brazitikos \cite{brazitikos2018polynomial} established the first polynomial bound on $\delta(d)$ (see also \cite{Brazitikos2016Diam}): 
using a sparsification result from \cite{BSS14}
 (see also \cite[Lemma 3.1]{barvinok_thrifty}) related to contact points of John's ellipsoid, he showed $\delta(d) \leq c d^{11/2}$ with an absolute constant $c > 0$. 
Recently, Dillon and Sober\'on \cite[Theorem 1.2]{dillon2020m} showed that a fractional version of 
Conjecture~\ref{conj:BKPdiam} holds.
 
 Since $v(1) = \delta(1) = 1,$ we will assume that 
 $d \geq 2$ throughout the paper. We use $[n]$ and $\binom{[n]}{k}$
to denote the sets $\{1, \dots, n\}$ and the set of all $k$-element subsets of $[n],$ respectively;   
for a family of   sets $\{K_1,\dots,K_n\}$ and
$\sigma \subset [n],$ $K_\sigma$ denotes the intersection
$\bigcap\limits_{i \in \sigma} K_i$.

We prove that $v(d)\leq(2d)^{3d}$ and $\delta(d)\leq(2d)^3$.
\begin{thm}
\label{thm:BKP_vol_diam_d3}
  Let $\{K_1,\dots,K_n\}$ be a family of  closed convex sets in $\Red$ such that their intersection 
 $K=K_1 \cap \dots \cap K_n$ is a convex body.  Then there is a 
$\mu\in\binom{[n]}{\leq2d}$ such that  
\[
\vol{d} K_\mu \leq (2d)^{3d} \vol{d} K
\quad \text{and} \quad   
\diam K_\mu \leq  (2d)^3 \diam K. 
\] 
\end{thm} 
The bound on $v(d)$ is not the best, as both \cite{naszodi2016proof} and \cite{brazitikos2017brascamp} provide stronger estimates. 
The method that yields it is new and quite simple as it does not require the use of the John ellipsoid. The bound on $\delta(d)$, on the other hand, is currently the best.

As we will see, Theorem~\ref{thm:BKP_vol_diam_d3} follows from our main result which concerns a very rough approximation of a convex polytope by the convex hull of $2d$ of its well-chosen vertices.
\begin{thm}
\label{thm:bkp_grunbaum_direct}
Let $\la > 0$ and let $Q \subset \Red$ be a convex body satisfying the inclusion
$Q \subset -\la Q$. Then
 \begin{enumerate}
  \item\label{item:bkp_grun_direct_2dplus} 
there is a subset $Q^\prime$ of  at most $2d+1$ extreme points of $Q$  such that
 \begin{equation*}
  Q \subset - (\la + 1)(d+1) \co Q^\prime,
 \end{equation*} and
   \item\label{item:bkp_grun_direct_2d}
there is a  subset $Q^{\prime \prime}$ of  at most $2d$ extreme points  of $Q$  such that
 \begin{equation*}
  Q \subset -(\la +1) (2d^2+2d+1) \co Q^{\prime \prime}.
 \end{equation*}
 \end{enumerate}
\end{thm}

By shifting the origin, one can guarantee $\lambda = d$ for any convex body $Q$ in $\Red$, see Lemma~\ref{lem:d_centers}.

Recall that the polar 
$K^\polar$ of a convex set $K \subset \Red$ is defined by
\[
K^\polar = \{ p  \in \Red \st \iprod{p}{x} \leq 1  \quad \text{ for all } x \in K \}.
\]
By a standard polarity (duality) argument, Theorem~\ref{thm:bkp_grunbaum_direct} yields the following.
\begin{thm}\label{thm:BanachMazurBKP}
 Let $\{K_1,\dots,K_n\}$ be a family of  closed convex sets in $\Red$ such that their intersection 
 $K=K_1 \cap \dots \cap K_n,$ is a convex body.  Then there is a point $z$ in the interior of $K$ 
 such that
 \begin{enumerate}
  \item $(K-z)^\polar \subset 
- \la (K-z)^\polar,$ where $1 \leq \la \leq d;$ 
  \item\label{item:BKPplusone} 
there is a $\sigma\in\binom{[n]}{\leq2d+1}$ such that
 \begin{equation*}
  K_{\sigma}-z \subset -(\lambda+1)(d+1) (K-z) \subset 
  -4d^2 (K-z);
 \end{equation*}
  \item\label{item:BKPtwod}
there is a $\mu\in\binom{[n]}{\leq2d}$ such that
 \begin{equation*}
  K_{\mu}-z \subset -(\la +1) (2d^2+2d+1) (K-z) \subset 
  -8 d^3 (K-z).
 \end{equation*}
 \end{enumerate}
\end{thm}
The containments in Theorem~\ref{thm:BanachMazurBKP} immediately yield
Theorem~\ref{thm:BKP_vol_diam_d3}.

There are several ways to find a point $z$ in the interior of a convex body 
$K$ such that $(K-z)^\polar \subset 
-d (K-z)^\polar$, see Lemma \ref{lem:d_centers}.

We may interpret Theorem~\ref{thm:BanachMazurBKP} as a new type of quantitative Helly-type theorem. 
First, for a set $A$ in $\Red$, we call $\tilde A=\mu A+x$ a \emph{homothet} of $A$, if $x$ is a vector in $\Red$ and $\mu\in\R$, $\mu\neq 0$.
Next, for two convex bodies $K$ and $L$ in $\Red$ with $K\subset L$, we define their \emph{homothetic distance} as the quantity
\[
\inf\{|\lambda|\st K-z \subset L-z\subset \lambda(K-z), z\in\Red, \lambda\in\R\}.
\]
Note that $\lambda$ may be negative. This definition is motivated by Gr\"unbaum's extension of the Banach--Mazur distance to non-symmetric convex bodies \cite{Gru63} (see also \cite{GLMP04,JN11} ). Now, Theorem~\ref{thm:BanachMazurBKP} can be rephrased as finding a subfamily of a family of convex bodies such that the intersection of the subfamily is at a bounded homothetic distance from the intersection of the entire family.

As a lower bound on $\delta(d)$, we show that the $\sqrt{d}$ in Conjecture~\ref{conj:BKPdiam} cannot be replaced by a lower power of $d$.
\begin{thm}\label{thm:bkp_diam_lower_bound}
For every $i \in [n]$, let $K_i = \{x \in \Red \st \iprod{x}{u_i} \leq 1\}$, where $u_i$ is a unit vector. 
Then for any $\sigma \subset [n],$  there is a point  in $K_{\sigma}$ with norm $\frac{d}{\sqrt{\abs{\sigma}}}$.
\end{thm}

It follows that if the $u_i$ form a sufficiently dense subset of the unit sphere (with a large $n$), then $K=K_{[n]}$ is almost the unit ball, while for any $\sigma\subset[n]$ of size $|\sigma|=2d$, we have that $\diam(K_{\sigma})\geq \sqrt{d/2}$.

We mention the following conjecture which is closely related to Theorem~\ref{thm:bkp_diam_lower_bound}. It can be found in a different formulation in \cite[p.194]{Bo04}.
\begin{conj}
\label{conj:2d_cups}
Let $\{u_1, \dots, u_{2d}\}$ be unit vectors  in $\Red$.
There is a point in the set  
\[
\bigcap\limits_{i=1}^{2d}\{x \in \Red \st \iprod{u_i}{x} \leq 1\}
\]
 with norm $\sqrt{d}$.
\end{conj}

\section{Proof of Theorem~\ref{thm:bkp_grunbaum_direct}}
Clearly, there is a largest volume simplex in $Q,$ whose vertices are extreme points of $Q.$ Let $S$ be such a simplex. 
It is well known that
\begin{equation}\label{eq:maxvolsimplexcontainment}
  -d(S-v)+v\supseteq Q,
 \end{equation}
where $v$ is the centroid of $S$.
If $v$ coincides with the origin, then by this inclusion, there is nothing to prove. Thus, we assume that $v$ is not origin. 

To prove assertion \eqref{item:bkp_grun_direct_2dplus}, we set $\ell$ to be the ray in $\Red$ emanating from the origin in the 
direction $-v$, and let $y$ be the 
point of intersection of $\ell$ with the boundary of $Q$.
 By Carath\'eodory's theorem, there are extreme points $q_1, \dots, q_d$
 of $Q$ on a support hyperplane to $Q$ at $y$ such that $y\in\cof{q_1, \dots, q_d}$.
 
 We set $Q^\prime$  to be  the union of the vertex set of $S$ and the set $\{q_1, \dots, q_d\}$.
 Since $v\in Q$ and by $Q \subset - \la Q$, one has 
 $-v \in \la Q$, which, by the choice of $Q^\prime$, yields 
 $-v\in  \la \co{Q^\prime}$, that is,
\begin{equation}\label{eq:vnotfar}
 v\in - \la \co{Q^\prime}.
\end{equation}
Consequently, 
\begin{equation*}
Q\subseteq
-d(S-v)+v\subseteq
-d(\co Q^\prime-v)+v=
-d \co Q^\prime+(d+1)v\subseteq
-(\lambda+1)(d+1)\co{Q^\prime}.
\end{equation*}
This completes the proof of assertion \eqref{item:bkp_grun_direct_2dplus} of the theorem.

We proceed with assertion \eqref{item:bkp_grun_direct_2d} of the theorem.
Our goal is to find a vertex of $S$ that can be omitted.

Consider the ray in $\Red$ emanating from $v$ in the direction $v$, and let $q$ denote the intersection point of this ray and the boundary of $S$. 
Let $q$ be in a facet $F$ of $S$. 
Denote  the simplex $\cof{v, F}$ by $S_1$ and its centroid by $w$. 
Clearly, $S\subseteq (d+1)(S_1-w)+w=(d+1)S_1-dw$, and hence,
\[
 -d(S-v)+v=-dS+(d+1)v\subseteq
 -d\left( (d+1)S_1-dw\right)+(d+1)S_1\subseteq
\]\[
-d(d+1)S_1+d^2w+(d+1)\left(-d(S_1-w)+w \right)=
\]\[
-2d(d+1)S_1+(2d^2+2d+1)w.
\]

Thus, by \eqref{eq:maxvolsimplexcontainment}, we have
\begin{equation}\label{eq:maxvolsimplexcontainmentsquare}
   Q\subseteq -2d(d+1)S_1+(2d^2+2d+1)w.
\end{equation}

Let $\ell_2$ be the ray in $\Red$ emanating from the origin in the 
direction $-w$, and let $y_2$ be the 
point of intersection of $\ell$ with the boundary of $Q$.
By Carath\'eodory's theorem, there are extreme points 
$p_1,\dots,p_d$ of $Q$ on a support hyperplane to $Q$ at $y_2$ such that $y_2 \in\cof{p_1, \dots, p_d}$.

We set $Q^{\prime\prime}$
to be the union of the vertex set of $F$ and the set 
$\{p_1, \dots, p_d\}$, and claim that $v \in \co{Q^{\prime\prime}}$.
Indeed, consider the simplex $\cof{o, F}.$ By construction,  $S_1 \subset \cof{o, F}.$  It follows that the ray emanating from the origin in the direction of $w$ intersects the facet $F$ of $S$. Thus, the origin is in $\co{Q^{\prime\prime}}$, and hence, 
$v \in \co{Q^{\prime\prime}}$, see Figure~\ref{fig:simplexinbody}.

\begin{figure}
 \includegraphics[width=0.3\textwidth]{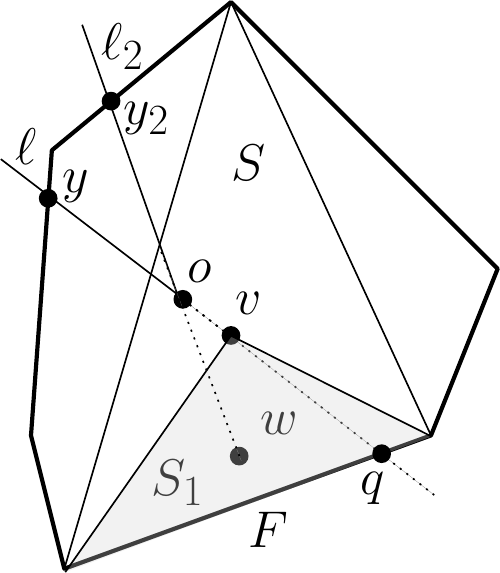}
 \caption{}\label{fig:simplexinbody}
\end{figure}

Since $q\in  \co{Q^{\prime\prime}}$, and $v$ is on the line segment $[o,q]$, we have $v\in \co{Q^{\prime\prime}}$. It follows that
\begin{equation}\label{eq:soneinp}
S_1\subset \co{Q^{\prime\prime}}. 
\end{equation}
Again, since $Q \subset - \la Q$,   we have $-w\in \la Q$, which, by the choice of $Q^{\prime\prime}$, yields $-w\in \la \co{Q^{\prime\prime}}$, that is,
\begin{equation}\label{eq:vnotfartwod}
 w\in - \la \co{Q^{\prime\prime}}.
\end{equation}

Finally, it follows from \eqref{eq:maxvolsimplexcontainmentsquare}, \eqref{eq:soneinp} and \eqref{eq:vnotfartwod} that
\begin{equation*}
Q\subseteq
-2d(d+1) \co{Q^{\prime\prime}} -\la (2d^2+2d+1)\co{Q^{\prime\prime}}\subseteq
-(\la +1) (2d^2+2d+1)\co{Q^{\prime\prime}}.  
\end{equation*}
This completes the proof of Theorem~\ref{thm:bkp_grunbaum_direct}.

\section{Proof of Theorem~\ref{thm:BanachMazurBKP}}\label{sec:BanachMazurBKP}

Recall that the \emph{centroid} $c(K)$ of  a body $K \subset \Red$ is defined by
\[
c(K) = \frac{1}{\vol{d}{K}}\int_K x \di x.
\] 

\begin{lem}
\label{lem:d_centers}
Let $K$ be a convex body in $\Red$.
Then the inclusion $(K-z)^\polar \subset 
-d (K-z)^\polar$ holds if $z$ is
\begin{enumerate}
\item\label{ass:centroid} the centroid of $K;$
\item\label{ass:santalo} the Santal\'o point of $K;$
\item\label{ass:john_center} the center of John ellipsoid of $K,$
\item\label{ass:lowner_center} the center of L\"owner ellipsoid of $K$.
\end{enumerate}
\end{lem} 
We note that other ``centers'' may be found using \cite[Theorem 5.1]{GLMP04}.

\begin{proof}
It is well known (see \cite{grunbaum1960partitions, Gru63}) that if $c(K)$ is the origin, then $K\subseteq -dK$. By taking the polar of this containment, we obtain assertion \eqref{ass:centroid}.

Recall that the \emph{Santal\'o point} of a convex body $K \subset \R^d$ is the unique point $z$  minimizing  
$\vol{d} (K-z) \vol{d} (K-z)^\polar.$ 
Next, \eqref{ass:santalo} follows from \cite[Proposition~D.2]{aubrun2017alice} which states that for any convex body $K$ whose Santal\'o point is the origin, the centroid of 
 $K^\polar$ is the origin.

Recall that if $E$ is the John ellipsoid of $K$, that is, the largest volume ellipsoid in $K$, and $E$ is centered at $z$, then $K\subseteq d(E-z)+z$. A similar statement holds for the L\"owner ellipsoid of $K$, that is, the smallest volume ellipsoid containing $K$.
These containments then easily yield assertions \eqref{ass:john_center} and \eqref{ass:lowner_center}, we leave the details to the reader.
\end{proof}

\begin{proof}[Proof of Theorem~\ref{thm:BanachMazurBKP}]
By Lemma~\ref{lem:d_centers}, there is a point $z$ in $K$ such that the inclusion 
 $(K-z)^\polar \subset 
- \la (K-z)^\polar$  holds with $1 \leq \la \leq d$.
Fix such $z$ and set $Q = (K-z)^\polar$.

Since the polar of the intersection of convex bodies containing the origin is the convex hull of the polars of these bodies, each extreme point of $Q$ belongs to a set 
$(K_i - z)^\polar$ for some $i \in [n]$. That is,
if $p$ is an extreme point of $Q$, then there exists $i$ 
such that 
\[
K_i - z \subset \{x \in \Red \st \iprod{p}{x} \leq 1\}.
\]
We will say in this case that $i$ corresponds to $p$.

Let $Q^\prime$ and $Q^{\prime \prime}$ be as in the assertions of Theorem~\ref{thm:bkp_grunbaum_direct}.
For every $p \in Q^\prime$, we find one index $i \in [n]$ that corresponds to $p$, and set $\sigma$ to be the set of these indexes. Clearly, $K_\sigma$ satisfies assertion \eqref{item:BKPplusone} of Theorem~\ref{thm:BanachMazurBKP}.
Analogously, for every $p \in Q^{\prime\prime},$ we find one index $i \in [n]$ that corresponds to $p$, and set $\mu$ to be the set of these indexes. Clearly, $K_\mu$ satisfies assertion \eqref{item:BKPtwod} of Theorem~\ref{thm:BanachMazurBKP}.
\end{proof}

\section{Lower bound for diameter}

In this section, we prove Theorem~\ref{thm:bkp_diam_lower_bound}. 
The result follows from the following lemma due to K. Ball and M. Prodromou.
\begin{lem}[\cite{Ball2009}, Theorem 1.4]
\label{lem:trace_bound_tight_frame}
Let vectors  $\{v_1,  \dots, v_n\}\subset \Red$ satisfying $\sum\limits_1^{n} v_i \otimes v_i = \id$.
Then for any positive semi-definite operator $T \colon \Red \to \Red,$
there is a point $p$ in the intersection of the strips 
$\{x \in \Red \st \abs{\iprod{x}{v_i}} \leq 1\}$ satisfying 
$\iprod{p}{T p} \geq \tr  {T}$.
\end{lem}

\begin{proof}[Proof of Theorem~\ref{thm:bkp_diam_lower_bound}]
We prove a bit stronger statement.  
We will find a point with the desired large norm in the subset 
\[
K_\sigma^\prime = \bigcap\limits_{i \in \sigma} \{x \st \abs{\iprod{u_i}{x}} \leq 1\} 
\]
of $K_{\sigma}$.
If $\{u_i \st i \in \sigma \}$ does not span the space, then $K_\sigma^\prime$ is unbounded. 
Thus, we assume $\{u_i \st i \in \sigma \}$ spans $\Red$. 
Consider  $A = \sum\limits_{i \in \sigma} u_i \otimes u_i$.
Since the vectors span the space, $A$ is positive definite.  
Using Lemma \ref{lem:trace_bound_tight_frame} with  
$ v_i = A^{-1/2} u_i,  i \in \sigma,$  
 and  $T  = A^{-1},$ we find a point $p$ in 
 \[
\bigcap\limits_{i \in \sigma} \{x \st \abs{\iprod{v_i}{x}} \leq 1\}.
\] such that
\[
\iprod{p}{A^{-1} p} \geq \tr{A^{-1}}.
\]

Denote $q = A^{-1/2} p$.
Then, by the choice of $p,$ 
\[
1 \geq \abs{\iprod{p}{A^{-1/2} u_i}} = 
\abs{\iprod{A^{-1/2}  p}{u_i}} = 
\abs{\iprod{q}{u_i}}. 
\] 
That is, $q \in K_\sigma^\prime$.
On the other hand, 
\[
\abs{q}^2 = \iprod{A^{-1/2} p}{A^{-1/2} p} = 
\iprod{p}{A^{-1} p } \geq \tr{A^{-1}}.
\]
Finally, since  $\tr A = \abs{\sigma}$ and by the Cauchy--Schwarz inequality,  
one sees that  
  $\tr A^{-1}$ is at least $\frac{d^2}{\abs{\sigma}}$.
Thus, $\abs{q}  \geq \frac{d}{\sqrt{\abs{\sigma}}}$.
\end{proof}


\bibliographystyle{amsalpha}
\bibliography{biblio}
\end{document}